\documentclass[12
pt,twoside]{amsart}
\usepackage{amssymb}
\usepackage{a4wide}
\usepackage[latin1]{inputenc}
\usepackage[T1]{fontenc}
\usepackage{times}

\def\hpq0{h^{p,q}_{\leq 0}}
\def\Hpq0{\H_{\leq 0}^{p,q}}

\def\dbar{\bar\partial}
\def\ddbar{\partial\dbar}

\def\R{{\mathbb R}}

\def\C{{\mathbb C}}

\def\H{{\mathcal H}}

\def\be{\begin{equation}}
\def\ee{\end{equation}}

\newtheorem{thm}{Theorem}[section]
\newtheorem{lma}[thm]{Lemma}

\newtheorem{prop}[thm]{Proposition}

\theoremstyle{definition}

\theoremstyle{remark}

\newtheorem{preremark}{Remark}
\newtheorem{preex}{Example}

\numberwithin{equation}{section}

\begin{document}

\begin{abstract} We give an estimate for the volume of an analytic variety (or more generally the mass of a positive closed current) close to a real submanifold $M$. Applications are given to the Hausdorff measure of the intersection of the variety with $M$ and the exponential integrability of plurisubharmonic functions on $M$. 
\end{abstract}

\title[]
{ Plurisubharmonic functions and real submanifolds of $\C^n$.}

\author[]{ Bo Berndtsson}

\bigskip

\maketitle

\quad \quad \quad \quad \quad \quad {\it To Jaap Korevaar, on his 100:th birthday}

\section{Introduction}

In the middle seventies, during my graduate studies, I spent a memorable semester in Amsterdam, with Jaap Korevaar as supervisor. I was working on a problem concerning a multivariate version of the M\"untz approximation theorem. Via the complex analytic approach to M\"untz' theorem, the problem amounts to give a sufficient condition for when a discrete subset of $\R^n$ is a set of uniqueness for the space of bounded holomorphic functions in a product of (right) half planes. I never managed to make much progress on this question. The best extension of the original work of Korevaar and Hellerstein, \cite{Korevaar-Hellerstein}, is still the paper of Ronkin, \cite{Ronkin}, and the optimal condition that we were looking for is still not known. 

There is only one part of my work in that context that I still think has some interest. That is a theorem saying that if a holomorphic  function has a lot of zeros on $\R^n$, then the area of its zero set in a neighbourhood of $\R^n$ also has to be big (see Theorem 2.1). The purpose of this note -- apart from contributing to the birthday volume in honor of Jaap -- is to give an extension of that result to more general submanifolds of $\C^n$ than $\R^n$. In the statement of the next theorem, which is the main result of this note, we use the notation $J$ for the complex structure on $\C^n$ and $\beta$ for the standard K\"ahler form on $\C^n$, $i\ddbar |z|^2$. 

\begin{thm} Let $B$ be the unit ball of $\C^n$ and let $M\subset B$ be a properly embedded smooth submanifold of codimension $m$. Assume that $M$ is generating in the sense that the complex span of its tangent space at every point is equal to all of $\C^n$
$$
T_p(M)+JT_p(M) =\C^n,
$$
for all $p\in M$. Let, for $t\leq t_0$ and $t_0$ fixed, sufficiently small, $U_t$ be the neighbourhood of $M$ consisting of points with distance less than $t$ from $M$. Let $\theta$ be closed positive current in $B$ and put for $r<1$
$$
\sigma(t)=\int_{U_t\cap rB} \theta\wedge \beta^{n-1}/(n-1)!
$$
Then 
$$
\sigma(t)\leq C_{r,M} |\theta| t^{m-1},
$$
where
$$
|\theta|=\int_{U_t\cap B} \theta\wedge \beta^{n-1}/(n-1)!.
$$
\end{thm}
My `Amsterdam theorem', alluded to above is a more precise version of this, saying roughly that $\sigma(t)/t^{m-1}$ is increasing, in case $M=\R^n\cap B$. There is a version of that in the general case as well, saying that $\sigma(t)/t^{m-1}$ is almost increasing, in the sense that for some constant (independent of $\theta, t, s$)
$$
\sigma(t)/t^{m-1}\leq C\sigma(s)/s^{m-1}
$$
if $t<s$. This is the theorem we will prove, and Theorem 1.1 is of course a consequence. Applied to the current of integration on a zero set (or `divisor'), $V$,  this also implies that if a holomorphic function has a large discrete set of zeros on $M$, then the volume of $V$ in neigbourhood of $M$ must also be large. In a similar way one can estimate the Hausdorff measure of $V\cap M$ from above.

We give the proof of the theorem in the next section, and there we also show why the condition that $M$ be generating is necessary when $M$ is a linear subspace.
In section 3 we 
then give some applications to the volume of zero sets . In section 4 we will give an application in a somewhat different direction, to the integrability of plurisubharmonic functions on the submanifold $M$, and give an alternative proof of a result of Duc-Viet Vu, \cite{Duc-Viet Vu}.

\medskip

{\bf Remark:}  As kindly pointed out to me by Duc-Viet Vu, Theorem 1.1  is essentially also contained in his paper \cite{Duc-Viet Vu}, although not explicitly stated there. More precisely, when the codimension $m$ equals $n$ (so that $M$ is totally real of maximal dimension)  the estimate in Theorem 1.1 is contained in the last line of the proof of his Proposition 4.5, and similar methods give the case of lower codimension too. Therefore the only novelty of this paper lies in the applications in section 3, and in the proof in section 4, which is different from the one in \cite{Duc-Viet Vu}.

\section{monotinicity properties of the mass of a positive current, and the proof of Theorem 1.1}

Let us first state the theorem from \cite{Berndtsson} mentioned in the introduction.
\begin{thm} Let $K$ be a compact convex subset of $\R^n$ and let $\theta$ be a positive closed current of bidimension $(p,p)$, defined in a neighbourhood of $K$ in $\C^n$. Let $d(z)$ be the distance from a point $z$ in $\C^n$ to $K$ considered as a subset of $\C^n$ and put
$$
\sigma(t):=\int_{ d<t}\theta\wedge \beta^{p}/p!
$$ 
(the mass of $\theta$  in the set $\{d<t\}$). Then
$$
\sigma(t)/t^{p}
$$
is increasing.
\end{thm}

We shall now discuss the more general situation of a compact subset of a submanifold $M$ of codimension $m$. We consider only the case $p=n-1$ and  assume $M$ is  {\it generating} in the sense that the complex span of its tangent space, $T_p(M)+JT_p(M)$, is equal to all of $\C^n$ for all $p\in M$. Clearly $\R^n$ is generating in this sense, as is any totally real submanifold of maximal dimension $n$, but $M$ can also be of higher dimension.

We first give an abstract version of the theorem. Let $v$ be a nonnegative  continuous plurisubharmonic function in some subdomain of $B$, such that $\{v<t\}$ is  relatively compact in $B$ and smoothly bounded for $t$ in a dense subset of the interval $(0,T)$. Let $u:=v^2/2$ and let $\theta$ be a closed positive current of bidegree $(1,1)$ in $B$. Put, for $1\leq m\leq n$
$$
\sigma(t)=\int_{v<t} \theta\wedge (i\ddbar u)^{m-1}\wedge\beta^{n-m}
$$
Then we have:
\begin{prop}
$$
\sigma(t)/t^{m-1}
$$
is increasing.
\end{prop}
\begin{proof}
In the proof we can assume that $\theta$ and $v$ are  smooth. Fix a value $t$ such that $\{v=t\}$ is smooth. Then Stokes' theorem gives that (the case $m=1$ is trivial)
$$
\sigma(t)=\int_{v=t} \theta\wedge (i\dbar u)\wedge (i\ddbar u)^{m-2}\wedge \beta^{n-m}=
t^{m-1}\int_{v=t} \theta\wedge (i\dbar v)\wedge (i\ddbar v)^{m-2}\wedge \beta^{n-m},
$$
since $\dbar u=t\dbar v$ and $\ddbar u=t\ddbar v+\partial v\wedge \dbar v$ when $v=t$. By Stokes' theorem again
$$
\sigma(t)/t^{m-1}=\int_{v<t}\theta \wedge (i\ddbar v)^{m-1}\wedge\beta^{n-m}.
$$
This is proved if $\{v=t\}$ is smooth. By assumption, this holds for a dense set of $t$, and taking limits we find that it holds for all $t$. Since the integrand is positive, it follows that the left hand side is (weakly) increasing. 
\end{proof}

The next step is to find a   function $u$ such that the integrand in the definition of $\sigma$ becomes comparable to the mass of $\theta$, or in other words
$$
(i\ddbar u)^{m-1}\wedge\beta^{n-m}
$$
is comparable to $\beta^{n-1}$.

We also need that 
 $u^{1/2}$  behaves roughly like the distance to $M$ close to $M$
 and
 $$
 (i\ddbar u^{1/2})^{m-1}\wedge\beta^{n-m}\geq0.
 $$

  For this we assume that $M$ is defined by $m$ equations, $\rho_j=0$, $j=1,...m$, where $\rho_j$ are smooth and satisfy $d\rho_1\wedge...d\rho_m\neq 0$ on $M$. The following lemma is well known. 
\begin{lma} $M$ is generating if and only if 
$$
\partial\rho_1\wedge...\partial\rho_m\neq 0,
$$
or, in other words, the differentials $\partial\rho_j$ are linearly independent at any point.
\end{lma}
\begin{proof} We only prove the `only if' direction since that is all we will use. Assume
$$
\sum(a_j+ib_j)\partial\rho=0.
$$
Then
$$
\sum a_j\partial\rho_j(v)=-i\sum b_j\partial\rho_j(v)
$$
for all vectors  $v$. Applying this to $Jv$ we get, since $\partial\rho(Jv)=i\partial\rho(v)$,
$$
\sum a_j\partial\rho_j(Jv)=\sum b_j\partial\rho_j(v),
$$
and taking real parts
$$
\sum a_jd\rho_j(Jv)=\sum b_j d\rho_j(v).
$$
This means  that if $v\in T(M)$, then both $v$ and $Jv$ lie in the annihilator of $\sum a_j d\rho_j$, so $M$ cannot be generating.
\end{proof}
Put
$$
w= (1/2)(\rho_1^2+...\rho_m^2);
$$
this is our first approximation of $u$.
Then
$$
i\ddbar w=\sum \rho_ji\ddbar\rho_j+i\sum\partial\rho_j\wedge\dbar\rho_j=: R+\omega
$$
(i.e. $R:=\sum \rho_ji\ddbar\rho_j$ and $\omega:=i\sum\partial\rho_j\wedge\dbar\rho_j$).
By the lemma, $\partial\rho_j$, $j=1, ...m$ span an $m$-dimensional space, $W$,  at any point $p$ in $M$. Fix a point $p$ and choose an orthonormal set of coordinates so that $dz_1, ...dz_m$ span $W$. Then 
$$
\omega\geq \delta\sum_1^m idz_j\wedge d\bar z_j=:\beta',
$$
for some $\delta>0$ that can be chosen uniform as the point varies in a compact subset of $M$. Let
$$
\beta''=\sum_{m+1}^n idz_j\wedge d\bar z_j,
$$
so that $\beta=\beta'+\beta''$.

Now consider the form
$$
\omega^{m-1}\wedge \beta^{n-m}
$$
which is larger than
$$
 \delta^{m-1} (\beta')^{m-1}\wedge\beta^{n-m}.
$$
\begin{lma}If $k<m$ then
$$
 (\beta')^{k-1}\wedge\beta^{n-k}\geq C_{n,m}\beta^{n-1}.
 $$
 \end{lma}
 \begin{proof} It is enough to prove this for $k=m-1$, since the left hand side gets larger when $k$ decreases. 
For this we decompose  $\beta=\beta'+\beta''$ and expand $\beta^{n-m}$ in the left hand side by the binomial theorem. Since $(\beta')^{m+1}=0$ we get only two terms
$$
 (\beta')^{m-1}\wedge(\beta'')^{n-m}, \quad \text{and} \,\, (\beta')^{m}\wedge(\beta'')^{n-m-1}.
 $$
 But, since also $(\beta'')^{n-m+1}=0$, these are precisely the terms we get when we expand $\beta^{n-1}$ in the right hand side which proves our claim. 
 \end{proof}
 
 Hence we have shown that on $M$
 $$
 \omega^{m-1}\wedge\beta^{n-m}\geq \delta'\beta^{n-1}.
 $$
 Since $R=0$ on $M$, we also have
 \be
 (i\ddbar w)^{m-1}\wedge\beta^{n-m}\geq \delta'\beta^{n-1},
 \ee
 and by continuity this also holds in a neighbourhood of $M$. 
 
 Thus $w$ satisfies the first of the requirements on the function $u$, but not the second. In fact, a direct computation shows that if $h:=w^{1/2}$, then
 $$
 i\ddbar h\geq -(2h)^{-1/}\sum\rho_ji\ddbar\rho_j,
 $$
 which is of the order of magnitude $-\beta$.  To compensate for that we put
 $$
 v:= h+Aw
 $$
 where $A$ is a sufficiently large constant, and let $\tilde u:=v^2$. It follows from Lemma 2.4 that
 $$ 
 (i\ddbar \tilde u^{1/2})^{m-1}\wedge\beta^{n-m}\geq 0
 $$
 if $A$ is sufficiently large. Moreover, one checks that 
 $$
(i\ddbar \tilde u)^{m-1}\wedge\beta^{n-m}
$$
is still comparable to $\beta^{n-1}$ if $v$ is close to zero, i. e. if we are close to $M$.
 
 With this we are almost ready to apply the proposition to get estimates for the mass of $\theta$ near $M$. The remaining problem is that the sets $\{v<t\}$ are not relatively compact.
 
 Fix a point in $M$ and let $B$ be a sufficiently small ball centered at the point, so that $M$ is defined by $m$ functions $\rho_j$ in a neighbourhood of  $B$. Let $f$ be a smooth convex function on $\C^n$ which is identically zero inside of $B$ and strictly positive outside, and let
 $$
 u:=\max (\tilde u, f^2).
 $$
 We can then apply Proposition 2.2 to this choice of $u$ and $t<t_0$ for some choice of $t_0>0$. Since $\sigma(t_0)$ is bounded by some constant, it follows that 
 $$
 \sigma(t)\leq Ct^{m-1}.
 $$
Inside the ball, $u=\tilde u$, so by the construction of $\tilde u$ Theorem 1.1 follows for this small ball. The general case is then obtained by a covering argument.

 Finally we show that the condition that $M$ be generating is necessary, at least when $M$ is a linear subspace of $\C^n$. Then $M$ can be written
 $$
 M=W\times S,
 $$
 where $W$ is the largest complex subspace contained in $M$ and $S$ is totally real. Say $\dim_{\C} W=k$ and $\dim_{\R}S=l$. The codimension, $m$,  of $M$  equals $2n-2k-l$. If $\theta$ is the current of integration on a complex hyperplane of $\C^n$ that contains $W$, then $\sigma(t)$ from Theorem 1.1 will be of the order $t^{l-1}$ as $t\to 0$. The complex dimension of  $M+JM$ is $k+l$, so if $M$ is not generating, then $k+l<n$. This implies $m>n-k>l$, so the theorem cannot hold.
 
 \section{Analytic varieties near generating submanifolds}
 In this section we look at the special case of positive closed currents defined by analytic varieties of codimension 1. A well known theorem of Lelong, \cite{Lelong}, says that if $f$ is a holomorphic function, then
 $$
 \theta=(1/2\pi) i\ddbar\log|f|^2
 $$
 is the current of integration on $V=\{f=0\}$ (with multiplicities). Moreover, the trace measure of $\theta$,
 $$
 \theta\wedge\beta^{n-1}/(n-1)!
 $$
 is the area measure on (the regular part of) $V$.
 
 We will now use Theorem 1.1 to estimate the size of the intersection of $V$ and $M$. 
 \begin{thm} Let $M$ be a generating submanifold of $\C^n$ and  $K$ a compact subset of $M$. Let $f$ be a holomorphic function defined in a neighbourhood $\Omega$ of $K$ in $\C^n$, and let $V=\{f=0\}$. Let $\{z_j\}_1^N$ be a set of points in $K\cap V$, such that $|z_j-z_k|>2\epsilon$ if $j\neq k$. If $\epsilon>0$ is sufficiently small, there is a constant $C$, not depending on $f$ or $\epsilon$ such that
 $$
 N\leq C|V| \epsilon^{m+1-2n},
 $$
 where $|V|$ is the total area of $V$ in the (fixed) neighbourhood $\Omega$.
 \end{thm}
 \begin{proof} We take $\epsilon$ so small that an $\epsilon$-neighbourhood of $K$ lies inside $\Omega$. Let $B_j$ be balls with center $z_j$ and radii $\epsilon$. By a well known theorem of Lelong, linear varieties minimize the area among all varieties containing the center of the ball. In other words,
 $$
 |V\cap B_j|\geq \pi^{n-1} \epsilon^{2n-2}/(n-1)!.
 $$
 Summing up we get that the area of $V$, i. e. the trace measure of $\theta$,  in an $\epsilon$-neighbourhood of $K$ is larger than $c_nN\epsilon^{2n-2}$. Let $t_0$ be the supremum of $t$ such that a $t$-neighbourhood of $K$ lies in $\Omega$. By Theorem 1.1
 $$
 C|V\cap\{d(z,K)<t_0\}|\geq c_nN\epsilon^{2n-2}\epsilon^{1-m}=c_n N \epsilon^{2n-1-m},
 $$
 which gives what the theorem claims. 
 \end{proof}
 As a consequence, we also get an estimate of the Hausdorff measure of $V\cap K$, in the same way as in \cite{Berndtsson}.
 \begin{thm} Let $H_p$ be the $p$-dimensional Hausdorff measure. Then, with the same notation and assumptions as in the previous theorem
 $$
 H_{2n-m-1}(V\cap K)\leq C |V|.
 $$
 \end{thm}
 \begin{proof}
 Choose $\epsilon$ as in the proof of the previous theorem and let $\{z_j\}_1^N$ be a maximal collection of points on $K\cap V$ such that $|z_j-z_k|>2\epsilon$, i. e. the balls $B_j$ are disjoint. Since the collection is maximal, the balls $B(z_j, 2\epsilon)$ cover $K\cap V$. Say the number of such points is $N_\epsilon$.  We get by the definition of Hausdorff measure that 
 $$
 \liminf_{\epsilon\to 0} N_\epsilon \epsilon^{2n-m-1}\geq c H_{2n-m-1}(K\cap V).
 $$
 Hence the claim follows from the previous theorem.
 \end{proof}
 When $M=\R^n$ this theorem was proved in \cite{Berndtsson}.  If  $f$ is a polynomial of degree $D$, it is well known that the  measure of $V$ inside  the unit ball is bounded by $A_nD$, where $A_n$ depends only on the dimension. Hence we get a similar upper bound for the Hausdorff measure of $V\cap M$. In the particular case when $M=\R^n$ we can scale this estimate and find that for $R>0$,
 $$
 |V\cap B_R\cap \R^n|\leq A_n R^{n-1},
 $$
 where $B_R$ is a ball of radius $R$.
 Estimates of this kind are also of interest in connection with estimates of the Hausdorff measure of nodal sets for eigenfunctions of the Laplacian, see \cite{Donnelly-Fefferman}.
 Note that the simple example $V=\{z_1=0\}$ and $M=\{y_1=...y_m=0\}$ (in $\C^n$ with coordinates $z_j=x_j+iy_j$) shows that the dimension of the Hausdorff measure is the right one.

 \section{Integrability of plurisubharmonic functions on $M$}
 
 In this section we will show how the following theorem from \cite{Duc-Viet Vu} follows from Theorem 1.1.
 \begin{thm} Let $\phi$ be a plurisubharmonic function defined in a neighbourhood of a compact subset $K$ of a smooth submanifold $M$ of $\C^n$. Assume that $M$ is generating. Then there is a constant $\alpha>0$ such that
 $$
 \int_K e^{-\alpha\phi} dM<\infty.
 $$
 \end{thm}
 
 We start by giving a general estimate for Newtonian potentials of certain measures, inspired by a proof of Skoda, \cite{Skoda}. 
Let $\mu\geq 0$ be a finite measure in $2B$ where $B$ is  the unit ball of $\C^n$. For each point $z$ in $B$, let
$$
\mu_z(s):= \int_{|\zeta-z|<s} d\mu(\zeta)
$$
for $s<1$. Put
$$
\nu_z(s)=\mu_z(s)/s^{2n-2}.
$$
We will be interested in measures $\mu$, such that $\nu_z(s)$ is increasing in $s$ for $s<1$, keeping in mind that by a well known result of Lelong, \cite{Lelong},  the Laplacian of any plurisubharmonic function has this property.

\begin{prop} Assume that the total mass of $\mu$ is bounded by 1, and that $\nu_z$ is increasing for a certain $z\in B$. Let
\be
U(z)=\int |z-\zeta|^{2-2n} d\mu(\zeta).
\ee
Then
\be
e^{\alpha U(z)/(2n-2)}\leq C_{n,\alpha}\int |z-\zeta|^{2-2n-\alpha}d\mu(\zeta),
\ee
for any $0<\alpha<1$.
\end{prop}
\begin{proof} First we note that
$$
U(z)=\int_{\{|z-\zeta|<1\}} |z-\zeta|^{2-2n} d\mu(\zeta)+\int_{\{|z-\zeta|\geq1\}}\ |z-\zeta|^{2-2n} d\mu(\zeta),
$$
and the second term in the right hand side is bounded by a constant. The first term can be written
$$
\int_0^{1} d\mu_z(s)/s^{2n-2}.
$$
Since $\mu_z=s^{2n-2}\nu_z$ we have
$$
d\mu_z=(2n-2)s^{2n-3}\nu_z ds+s^{2n-2}d\nu_z,
$$
from which we get
\be
d\mu_z/s^{2n-2+\alpha}=(2n-2)s^{-1-\alpha}\nu_z ds+s^{-\alpha}d\nu_z
\ee
and
\be
d\mu_z/s^{2n-2}=(2n-2)s^{-1}\nu_z ds+d\nu_z
\ee
We may assume that the right hand side in (4.2) is finite. This implies by (4.3) that
\be
\int_0^1 d\nu_z(s)/s^\alpha <\infty\,\,\text{and}\,\, \int_0^1 s^{-1-\alpha}\nu_z(s)ds<\infty.
\ee
From (4.4) we get
$$
U(z)\leq C+\int_0^{1} d\mu_z(s)/s^{2n-2}\leq C +(2n-2)\int_0^1 \nu_z(s) ds/s,
$$
and
integrating by parts we find
$$
\int_0^1\nu_z(s)ds/s= \liminf_{s\to 0} -\nu_z(s)\log s +\int_0^{1} \log(1/s)d\nu_z(s)= \int_0^{1} \log(1/s)d\nu_z(s)
$$
(the $\liminf$-part must vanish by the second part of (4.5)). 
Since the total integral of $d\nu_z$ is less than $\nu_z(1)$, and the total mass of $\mu_z$ is at most 1, the integral of  
$d\nu_z$ is bounded by 1. Hence it follows from Jensen's inequality that
$$
e^{\alpha U(z)/(2n-2)}\leq C \int_0^1 d\nu_z(s)/s^{\alpha},
$$
which by (4.3) is dominated by
$$
C\int_0^1 d\mu_z(s)/s^{2n-2+\alpha},
$$
which is smaller than the right hand side of (4.2). This completes the proof of the proposition.

\end{proof}
Note that it follows from the proposition that $e^{\alpha U(z)/(2n-2)}$ is  integrable over $B$ if $\nu_z(s)$ is increasing for all $z$ in $B$. We shall now couple this argument with Theorem 1.1 to get a similar conclusion for integrals over $M$ when $M$ is generating.

\begin{lma} Let $K$ be a compact subset of $M$ and let $d(\zeta)$ be the distance from $\zeta$ to $M$. Denote by $dM$ surface measure on $M$. Then
$$
I:=\int_K |z-\zeta|^{-(2n-2+\alpha)} dMu(z)\leq A_0 d^{2-m-\alpha}(\zeta) + A_1,
$$
for some constants $A_j$ if $\zeta$ is sufficiently close to $M$.
\end{lma}
\begin{proof} Let
$$
\tau(r)=\int_{|z-\zeta|<r} dM(z).
$$
Then $\tau(r)=0$ if $r<d(\zeta)$ and $\tau(r)\leq C r^{2n-m}$, if $r$ is smaller than some small constant, $a$. Hence
$$
I\leq A+\int_d^{a} d\tau(r)/r^{2n-2+\alpha}=A+A'\int_d^{a} \tau(r)d\tau/r^{2n-1+\alpha}\leq
$$
$$
A+CA'\int_d^{a} dr/r^{m+\alpha-1}=A_1+A_0 /d^{m-2+\alpha}.
$$
\end{proof}
Theorem 4.1 is now a simple consequence. A plurisubharmonic function $\phi$ can be decomposed into a harmonic part and a Newtonian potential, and for integrability question we only have to worry about the potential. After rescaling and multiplication by a constant, the potential has the form of the function $U$ in Proposition 4.2. When integrationg $e^{\alpha U/(2n-2)}$ over a compact subset of $M$ we are by Proposition 4.2 and Lemma 4.3 reduced to studying integrals of the form
$$
\int_{B} d\mu(\zeta)/d^{m-2+\alpha}.
$$
Here the measure $\mu$ is the Laplacian of $\phi$ which is also the trace measure, $\theta\wedge\beta^{n-1}/(n-1)!$ of the positive closed current $\theta=i\ddbar\phi$.
These integrals are finite if $\alpha<1$ by Theorem 1.1 and an integration by parts argument like the ones we have already used several times, and this completes the proof of Theorem 4.1.

\end{document}